\documentclass[a4paper, 11pt]{amsart}
\usepackage{amsmath,amssymb,amscd}
\usepackage{amsfonts}
\usepackage[english]{babel}
\usepackage[leqno]{amsmath}
\usepackage{amssymb,amsthm}
\usepackage{mathrsfs} 
\usepackage{amscd}
\usepackage{enumerate}
\usepackage{epsfig}
\usepackage{relsize}
\usepackage{layout}
\usepackage{fullpage}
\usepackage[usenames,dvipsnames]{xcolor}
\usepackage[backref=page]{hyperref}
\usepackage{tikz}
\hypersetup{
 colorlinks,
 citecolor=Green,
 linkcolor=Red,
 urlcolor=Blue}
\usepackage[matrix,arrow,tips,curve]{xy}
\input{xy}
\xyoption{all}
\definecolor{darkgreen}{rgb}{0.0, 0.7, 0.0}
\definecolor{purple}{rgb}{0.5, 0.0, 0.5}
\definecolor{red}{rgb}{0.8, 0.2, 0.0}

\newtheorem{thm}{Theorem}[section]
\newtheorem{lemma}[thm]{Lemma}

\newtheorem{prop}[thm]{Proposition}
\newtheorem{cor}[thm]{Corollary}

\numberwithin{equation}{section}
\theoremstyle{definition}
\newtheorem{defi}[thm]{Definition}
\newtheorem{bthm}{Theorem}

\theoremstyle{remark}
\newtheorem{remark}[thm]{Remark}

\newcommand{\Z}{\mathbb{Z}}
\newcommand{\Q}{\mathbb{Q}}
\newcommand{\R}{\mathbb{R}}
\newcommand{\NN}{\mathbb{N}}

\DeclareMathOperator{\Supp}{Supp}

\def \PP{\mathbb{P}}
\def \ZZ{\mathbb{Z}}

\def\I{\mathcal I}

\def \E{\mathcal E}
\def\O{\mathcal O}

\def\M0{\mathcal M^0}

\DeclareMathOperator{\Sing}{{Sing}}

\DeclareMathOperator{\Int}{{Int}}

\def\B{\mathbf{B}}

\newcommand{\rk}{\operatorname{rk}}

\newcommand{\Eff}{\operatorname{Eff}}

\newcommand{\Nef}{\operatorname{Nef}}
\newcommand{\Amp}{\operatorname{Amp}}

\begin{document}

\title[Augmented and restricted base loci of cycles]{Augmented and restricted base loci of cycles}

\author[A.F. Lopez]{Angelo Felice Lopez}

\thanks{Research partially supported by INdAM (GNSAGA) and by the MIUR national projects ``Geometria delle variet\`a algebriche" PRIN 2010-2011 and ``Spazi di moduli e applicazioni" FIRB 2012.}

\address{\hskip -.43cm Dipartimento di Matematica e Fisica, Universit\`a di Roma Tre, Largo San Leonardo Murialdo 1, 00146 Roma, Italy. e-mail {\tt lopez@mat.uniroma3.it}}

\begin{abstract}
We introduce augmented and restricted base loci of cycles and we study the positivity properties naturally defined by these base loci.
\end{abstract}

\maketitle

\section{Introduction}
\label{intro}

One of the most important facts in algebraic geometry is that the geometry of a variety is reflected in the geometry of its subvarieties. There is, however, a big difference between codimension one subvarieties and higher codimensional ones. In the first case several tools are at hand, such as linear systems, ample divisors, vanishing theorems and so on. On the other hand, no similar tool is available in the study of higher codimensional cycles and this certainly makes the theory harder. Well-known famous problems are still open in that case, such as the Hodge conjecture or Grothendieck standard conjectures.

When dealing with algebraic cycles one can define effective, pseudoeffective and big cycles, but perhaps a good notion of positive cycles still lacks \cite[Problem 6.13]{delv}. A few years ago, Ottem \cite{o1, o2} defined the notion of ample subschemes and proved several beautiful results about them. On the other hand, the properties of the cone generated by them remain mysterious,  for example it is not known whether one can "move" multiples of ample subschemes or if they are big (except when $k=1$ \cite{o2} or $n-1$ \cite{o1}). Perhaps another difficulty is that higher codimensional nef cycles may not be pseudoeffective \cite{delv, o3}. More recently, in a series of papers, Fulger and Lehmann \cite{fl1, fl2, fl3, fl4} laid out a general theory of cones of cycles, by introducing and studying several notions of positivity of cycles.

In the present paper we take a different approach. We observe that, in the case of Cartier divisors, there are well-established notions of base loci, such as the stable, augmented or restricted base locus \cite{elmnp1, elmnp2} and that positivity properties of divisors are precisely reflected in their base loci. For example a divisor $D$ is ample if and only if $\B_+(D)=\emptyset$, it is nef if and only if $\B_-(D)=\emptyset$, it is big if and only if $\B_+(D)$ is not the whole variety. Our goal in the present paper is to take the same path in the case of cycles.  

The starting observation is that $\B_-(D)$ and $\B_+(D)$ can also be described using the numerical base locus of perturbations of $D$, that is the intersection of the supports of the effective divisors numerically equivalent to perturbations of $D$ (see Lemma \ref{divis}). On the other hand, this process can now be carried over to cycles.

Let $X$ be a projective variety of dimension $n$ and let $k$ be an integer such that $1 \le k \leq n-1$. We denote by $Z_k(X)_{\R}$ the vector space of real $k$-cycles, $N_k(X)$ the vector space of real $k$-cycles modulo numerical equivalence and by $[Z]$ the numerical equivalence class of a real $k$-cycle $Z$ on $X$ (for definitions see section 2).

\begin{defi}
\label{luoghibase}
Let $\alpha \in N_k(X)$. Set
\[ |\alpha|_{\rm num} = \{ e \in Z_k(X)_{\R} : e \ \mbox{is effective and} \ [e] = \alpha\}. \]
The {\it numerical stable base locus} of $\alpha$ is
\[ \B_{\rm num}(\alpha)= \begin{cases} \ \ \ \ \ X & {\rm if} \ |\alpha|_{\rm num} = \emptyset \\ \bigcap\limits_{\ \ \ e \in |\alpha|_{\rm num}} {\rm Supp}(e) & {\rm if} \ |\alpha|_{\rm num} \neq \emptyset \end{cases} . \]
The {\it augmented base locus} of $\alpha$ is 
\[ \B_+(\alpha)= \bigcap\limits_{A_1, \ldots, A_{n-k}} \B_{\rm num}(\alpha - [A_1 \cdots A_{n-k}]) \]
and the {\it restricted base locus} of $\alpha$ is 
\[ \B_-(\alpha)= \bigcup\limits_{A_1, \ldots, A_{n-k}} \B_{\rm num}(\alpha + [A_1 \cdots A_{n-k}]) \]
where $A_1, \ldots, A_{n-k}$ run among all ample $\R$-Cartier $\R$-divisors on $X$. 
\end{defi}

As a matter of fact, rather than perturbing with complete intersection of ample divisors, one can use push-forwards $f_* [A_1 \cdots A_{n-k}]$ under finite flat morphisms $f : Y \to X$, with $A_1, \ldots, A_{n-k}$ ample on $Y$. 
As we will see in section \ref{pf}, if we let $PCI_k(X)$ be the convex cone generated by those classes, one has (see Remark \ref{pert}) that
\[ \B_+(\alpha)= \bigcap\limits_{\gamma \in PCI_k(X)} \B_{\rm num}(\alpha - \gamma) \ \hbox{and} \ \B_-(\alpha)= \bigcup\limits_{\gamma \in PCI_k(X)} \B_{\rm num}(\alpha + \gamma). \]

As mentioned above, these loci do coincide, in the case of Cartier divisors, with their counterparts. On the other hand, a basic question arises: What are the positivity properties of $\alpha \in N_k(X)$ when $\B_+(\alpha)$ or  $\B_-(\alpha)$ are empty or properly contained in $X$?

To answer these questions we introduce the following positivity property of cycles. As we will see, it can also be considered a partial answer to \cite[Problem 6.13]{delv}. See also Lemma \ref{coni2} and  Remark \ref{comp} for a comparison with other positivity notions.
  
\begin{defi}
Let $X$ be a projective variety of dimension $n$ and let $k$ be an integer such that $1 \le k \leq n-1$. Set
\[ P_k(X) = \{\alpha \in N_k(X): \exists \ A_1, \ldots, A_{n-k} \ \mbox{ample $\R$-Cartier $\R$-divisors on $X$ and} \]
\[ \ \  \exists \ \beta \in N_k(X) \ \mbox{with} \ \B_{\rm num} (\beta) = \emptyset \ \mbox{and} \ \alpha = [A_1 \cdots A_{n-k}] + \beta \}. \]
\end{defi}
(see also Proposition \ref{coni3} for the analogous formulation in terms of $PCI_k(X)$.)

\vskip .3cm

Our first result is the ensuing

\begin{bthm}
\label{aperto}
Let $X$ be a projective variety of dimension $n$ and let $k$ be an integer such that $1 \le k \leq n-1$.
Then $P_k(X)$ is a convex cone in $N_k(X)$. Moreover suppose that $X$ is smooth. Then $P_k(X)$ is open and full-dimensional.
\end{bthm}

We observe that a variety may have Picard rank one and therefore, in some sense, the opennes of $P_k(X)$ is not at all accounted for by complete intersection of ample divisors.

Our answer to the above questions is given in the following two results. 

As for the augmented base locus we have

\begin{bthm}
\label{b+}
Let $X$ be a projective variety of dimension $n$, let $k$ be an integer such that $1 \le k \leq n-1$ and let $\alpha \in N_k(X)$.
Then 
\begin{itemize}
\item[(i)] $\B_+(\alpha) \subsetneq X$ if and only if $\alpha$ is big; 
\item[(ii)] $\B_+(\alpha) = \emptyset$ if and only if $\alpha \in P_k(X)$. 
\end{itemize}
\end{bthm}

We observe that while (i) is the same as in the case of divisors, (ii) introduces a novel positivity property of cycles which, in some sense, resembles ampleness of divisors.

As for the restricted base locus we have

\begin{bthm}
\label{b-}
Let $X$ be a projective variety of dimension $n$, let $k$ be an integer such that $1 \le k \leq n-1$ and let $\alpha \in N_k(X)$.
Then 
\begin{itemize}
\item[(i)] If $\B_-(\alpha) \subsetneq X$, then $\alpha$ is pseudoeffective;
\item[(ii)] If $\alpha$ is pseudoeffective and the base field is uncountable, then $\B_-(\alpha) \subsetneq X$; 
\item[(iii)] If $\B_-(\alpha) = \emptyset$, then $\alpha \in \overline{P_k(X)}$;
\item[(iv)] If $X$ is smooth and $\alpha \in \overline{P_k(X)}$, then $\B_-(\alpha) = \emptyset$. 
\end{itemize}
\end{bthm}

Again (i) and (ii) resemble the case of divisors, while (iii) and (iv) give more information on the cone $P_k(X)$.

We would like to thank J.C. Ottem for several helpful conversations.

We also thank the referee for the big contribution given to improve the paper. 

\section{Notation} 
\label{not}

A {\it variety} is by definition an integral separated scheme of finite type over a field.

Throughout the paper $X$ will be a projective variety of dimension $n \ge 2$ defined over an arbitrary algebraically closed field and, unless otherwise specified, $k$ will be an integer such that $1 \le k \leq n-1$. In some cases we will require that $X$ is smooth. Whenever countability arguments are required, in Theorem \ref{b-}(ii), Proposition \ref{stab}, Remark \ref{altro} and Corollary \ref{dense}, we will need the base field to be uncountable. 

Let ${\rm deg}:A_0(X) \to \Z$ be the group homomorphism that sends any point to $1$. A $k$-cycle $Z$ is said to be {\it numerically trivial} if $\deg(P(E_I) \cap Z) = 0$ for any weight $k$ homogeneous polynomial $P (E_I )$ in Chern classes of a finite set of vector bundles on X (see \cite[Def.~19.1]{fu}). The quotient of $Z_k(X)$ by the numerically trivial cycles is denoted by $N_k(X)_{\Z}$; this is a free abelian group of finite rank by \cite[Ex.~19.1.4]{fu}. We set $N_k(X) = N_k(X)_{\ZZ} \otimes_{\Z}  {\R}$.

The cone of {\it effective} $k$-cycles will be denoted by $\Eff_k(X)$; the cone of {\it pseudoeffective} $k$-cycles is the closure $\overline{\Eff_k(X)}$ and the cone of {\it big} $k$-cycles is ${\rm Big}_k(X) := \Int(\overline{\Eff_k(X)})$.

\section{The case of divisors} 
\label{div}

In this section we verify that, in the case of divisors, the definitions of augmented and restricted base loci can equivalently be given using numerical base loci. 

\begin{defi}
\label{bl}
Let $D$ be an $\R$-Cartier $\R$-divisor on $X$. Set
\[ |D|_{\sim \R} = \{ E \ \R\mbox{-Cartier} \ \R\mbox{-divisor on $X$} : E \ \mbox{is effective and} \ E \sim_{\R} D \}. \]
\[ |D|_{\rm num} = \{ E \ \R\mbox{-Cartier} \ \R\mbox{-divisor on $X$} : E \ \mbox{is effective and} \ E \equiv D \}. \]
The {\it stable base locus} of $D$ is
\[ \B(D)= \begin{cases} \ \ \ \ \ X & {\rm if} \ |D|_{\sim \R} = \emptyset \\ \bigcap\limits_{\ \ \ E \in |D|_{\sim \R}} {\rm Supp}(E) & {\rm if} \ |D|_{\sim \R} \neq \emptyset \end{cases} . \]
The {\it numerical stable base locus} of $D$ is
\[ \B_{\rm num}(D) = \begin{cases} \ \ \ \ \ X & {\rm if} \ |D|_{\rm num} = \emptyset \\ \bigcap\limits_{\ \ \ E \in |D|_{\rm num}} {\rm Supp}(E) & {\rm if} \ |D|_{\rm num} \neq \emptyset \end{cases} . \]
The {\it augmented base locus} of $D$ is 
\[ \B_+(D) = \bigcap\limits_{A} \B(D - A) \]
and the {\it restricted base locus} of $D$ is 
\[ \B_-(D) = \bigcup\limits_{A} \B(D + A) \]
\noindent where $A$ runs among all ample $\R$-Cartier $\R$-divisors on $X$. 
\end{defi}
The above definitions of $\B_+(D)$ and $\B_-(D)$ concide with the ones in \cite{elmnp1}. For $\B_+(D)$ it is straightforward that Definition \ref{bl} is the same as \cite[Def.~1.2]{elmnp1}. Given this, for $\B_-(D)$, to prove the equivalence, one can use \cite[Lemma 1.14]{elmnp1} and Proposition \ref{tutte}\eqref{unb+}.

We recall that in \cite[Ex.~1.16]{elmnp1} is shown that $\B_-(D) \subseteq \B_+(D)$ and if $D$ is a $\Q$-Cartier $\Q$-divisor, then $\B_-(D) \subseteq \B(D) \subseteq \B_+(D)$. 

As for the relation with numerical base loci, we have
 
\begin{lemma}
\label{divis}
Let $D$ be an $\R$-Cartier $\R$-divisor on $X$.
Then
\begin{itemize}
\item[(i)] $\B_-(D) \subseteq \B_{\rm num}(D) \subseteq \B(D) \subseteq \B_+(D)$; 
\item[(ii)] $\B_+(D)= \bigcap\limits_{A} \B_{\rm num} (D - A)$,
\item[(iii)] $\B_-(D)= \bigcup\limits_{A} \B_{\rm num} (D + A)$ 
\end{itemize}
where in (ii) and (iii) $A$ runs among all ample $\R$-Cartier $\R$-divisors on $X$.
\end{lemma}
\begin{proof}
This is straightforward.
\end{proof}

In particular, when Cartier and Weil divisors coincide, the two notions of base loci, associated to a Cartier divisor and to its class, are the same.

\begin{cor}
\label{divisori}
Let $X$ be a locally factorial projective variety, let $D$ be a Weil $\R$-divisor on $X$ and let $[D] \in N_{n-1}(X)$.
Then
\begin{itemize}
\item[(i)] $\B_{\rm num}(D) = \B_{\rm num}([D])$;
\item[(ii)] $\B_+(D) = \B_+([D])$;
\item[(iii)] $\B_-(D)= \B_-([D])$.
\end{itemize}
\end{cor}
\begin{proof}
Follows by the definitions and Lemma \ref{divis}.
\end{proof}

\section{Properties of base loci of cycles} 
\label{prbase}

We collect in one single statement some basic properties. They are analogues of similar results in \cite{elmnp1}.

\begin{prop}
\label{tutte}
Let $\alpha, \beta \in N_k(X)$ and let $A_1, \ldots, A_{n-k}$ be ample $\R$-Cartier $\R$-divisors on $X$. Then 
\begin{enumerate}[(i)]
\item \label{unione-i} $\B_{\rm num}(\alpha) = \B_{\rm num}(b \alpha)$ for every $b \in \R^+$;
\item \label{unione-ii} $\B_{\rm num}(\alpha + \beta) \subseteq \B_{\rm num}(\alpha) \cup \B_{\rm num}(\beta)$;
\item \label{nest} $\B_-(\alpha) \subseteq \B_{\rm num}(\alpha) \subseteq \B_+(\alpha)$;
\item \label{riduzgen}  For any ample $\R$-Cartier $\R$-divisors $A'_1, \ldots, A'_{n-k}$ on $X$ there exists an $\varepsilon_0 > 0$ such that for every $0 < \varepsilon \leq \varepsilon_0$ we have
\begin{enumerate}[(a)] 
\item \label{riduzgen-a} $\B_{\rm num} ([A_1 \cdots A_{n-k}] - \varepsilon [A'_1 \cdots A'_{n-k}]) = \emptyset$; 
\item \label{riduzgen-b} $\B_+(\alpha) = \B_{\rm num}(\alpha - \varepsilon [A'_1 \cdots A'_{n-k}])$.
\end{enumerate}
\item \label{unione2} $\B_+(\alpha + \beta) \subseteq \B_+(\alpha) \cup \B_+(\beta)$;
\item \label{numerabile} $\B_-(\alpha) = \bigcup\limits_{m \in \NN^+} \B_{\rm num}(\alpha + \frac{1}{m} [A_1 \cdots A_{n-k}])$;
\item \label{unb+} $\B_-(\alpha) = \bigcup\limits_{A'_1, \ldots, A'_{n-k}} \B_+(\alpha + [A'_1 \cdots A'_{n-k}])$
where $A'_1, \ldots, A'_{n-k}$ run among all ample $\R$-Cartier $\R$-divisors on $X$. 
\end{enumerate}
\end{prop}
\begin{proof}
The proof of \eqref{unione-i} and \eqref{unione-ii} is straightforward.

To see \eqref{nest}, since $\B_{\rm num} ([A_1 \cdots A_{n-k}]) = \emptyset$ by \eqref{unione-ii}, we have that
\[ \B_{\rm num} (\alpha + [A_1 \cdots A_{n-k}]) \subseteq \B_{\rm num} (\alpha) \subseteq \B_{\rm num} (\alpha - [A_1 \cdots A_{n-k}]) \]
and then \eqref{nest} follows by Definition \ref{luoghibase}.

To see \eqref{riduzgen} observe that for $i = 1, \ldots, n-k$ we can write 
\[ A_i = \sum\limits_{j = 1}^{s_i} c_{ij} A_{ij} \ \mbox{and} \ A'_i = \sum\limits_{l = 1}^{t_i} d_{il} A''_{il} \]
with $c_{ij}, d_{il} \in \R^+$, $A_{ij}, A''_{il}$ ample Cartier divisors. Let $m \gg 0$ be such that
$D_{ijl} := m A_{ij} - A''_{il}$ is ample for every $i, j, l$. Now
\[ [A'_1 \cdots A'_{n-k}] = \sum\limits_{l_1 \in \{1, \ldots, t_1\}, \ldots, l_{n-k} \in \{1, \ldots, t_{n-k}\}} d_{1l_1} \cdots   d_{n-k,l_{n-k}} [A''_{1l_1} \cdots A''_{n-k,l_{n-k}}] \]
and, setting $t = t_1 \cdots t_{n-k}$
\[ [A_1 \cdots A_{n-k}] = \sum\limits_{j_1 \in \{1, \ldots, s_1\}, \ldots, j_{n-k} \in \{1, \ldots, s_{n-k}\}} c_{1j_1} \cdots  c_{n-k,j_{n-k}} [A_{1j_1} \cdots A_{n-k,j_{n-k}}]  = \]
\[ = \sum\limits_{j_1, \ldots, j_{n-k}} c_{1j_1} \cdots  c_{n-k,j_{n-k}} \frac{1}{t^{n-k}}[(tA_{1j_1}) \cdots (tA_{n-k,j_{n-k}})]  = \]
\[ = \sum\limits_{j_1, \ldots, j_{n-k}} c_{1j_1} \cdots  c_{n-k,j_{n-k}} \frac{1}{t^{n-k}}[(\sum\limits_{l_1, \ldots, l_{n-k}}A_{1j_1}) \cdots (\sum\limits_{l_1, \ldots, l_{n-k}}A_{n-k,j_{n-k}})]  = \]
\[ = \sum\limits_{j_1, \ldots, j_{n-k}} c_{1j_1} \cdots  c_{n-k,j_{n-k}} \frac{1}{t^{n-k}}[(\frac{1}{m} \sum\limits_{l_1, \ldots, l_{n-k}}(A''_{1l_1} + D_{1j_1l_1})) \cdots (\frac{1}{m} \sum\limits_{l_1, \ldots, l_{n-k}}(A''_{n-k,l_{n-k}} + D_{n-k,j_{n-k},l_{n-k}}))]  = \]
\[ = \sum\limits_{j_1, \ldots, j_{n-k}} c_{1j_1} \cdots  c_{n-k,j_{n-k}} \frac{1}{(tm)^{n-k}}[\sum\limits_{l_1, \ldots, l_{n-k}}A''_{1l_1} \cdots A''_{n-k,l_{n-k}}] + \beta \]
where $\beta$ is sum of intersections of $A''_{il}$ and $D_{ijl}$. Let $\varepsilon_0 > 0$ be such that
\[  \sum\limits_{j_1, \ldots, j_{n-k}} \frac{c_{1j_1} \cdots  c_{n-k,j_{n-k}}}{(tm)^{n-k}} - \varepsilon_0 d_{1l_1} \cdots   d_{n-k,l_{n-k}} > 0 \ \mbox{for every} \ l_1, \ldots, l_{n-k}. \]
Then 
\[ [A_1 \cdots A_{n-k}] - \varepsilon_0 [A'_1 \cdots A'_{n-k}] = \]
\[ = \sum\limits_{l_1, \ldots, l_{n-k}}\big( \sum\limits_{j_1, \ldots, j_{n-k}} \frac{c_{1j_1} \cdots  c_{n-k,j_{n-k}}}{(tm)^{n-k}} - \varepsilon_0 d_{1l_1} \cdots   d_{n-k,l_{n-k}} \big) [A''_{1l_1} \cdots A'_{n-k,l_{n-k}}] + \beta \]
Hence $[A_1 \cdots A_{n-k}] - \varepsilon_0 [A'_1 \cdots A'_{n-k}]$ is a sum of intersections of ample $\R$-Cartier $\R$-divisors and so is $[A_1 \cdots A_{n-k}] - \varepsilon [A'_1 \cdots A'_{n-k}]$. Therefore
$\B_{\rm num} ([A_1 \cdots A_{n-k}] - \varepsilon [A'_1 \cdots A'_{n-k}]) = \emptyset$ by \eqref{unione-ii}. This proves \eqref{riduzgen}(a).

To see \eqref{riduzgen}(b) choose $A_{i1}, \ldots, A_{i, n-k}$, $1 \leq i \leq s$, ample $\R$-Cartier $\R$-divisors such that 
 \[ \B_+(\alpha)= \bigcap\limits_{i=1}^s \B_{\rm num}(\alpha - [A_{i1} \cdots A_{i, n-k}]). \]
 By \eqref{riduzgen}(a) we can choose $\varepsilon_0 > 0$ such that for every $0 < \varepsilon \leq \varepsilon_0$ we have 
 \[ \B_{\rm num} ([A_{i1} \cdots A_{i,n-k}] - \varepsilon [A'_1 \cdots A'_{n-k}]) = \emptyset \ \mbox{for all} \ 1 \leq i \leq s. \]
 Therefore, using \eqref{unione-ii}, for all $1 \leq i \leq s$ we have
 \[ \B_{\rm num}(\alpha - \varepsilon [A'_1 \cdots A'_{n-k}]) = \B_{\rm num}(\alpha - [A_{i1} \cdots A_{i, n-k}] + [A_{i1} \cdots A_{i, n-k}] - \varepsilon [A'_1 \cdots A'_{n-k}]) \subseteq  \]
\[ \subseteq \B_{\rm num}(\alpha - [A_{i1} \cdots A_{i, n-k}]) \cup \B_{\rm num}([A_{i1} \cdots A_{i, n-k}] - \varepsilon [A'_1 \cdots A'_{n-k}]) =  \B_{\rm num}(\alpha - [A_{i1} \cdots A_{i, n-k}])\]
hence $\B_{\rm num}(\alpha - \varepsilon [A'_1 \cdots A'_{n-k}]) \subseteq \B_+(\alpha)$. Since the other inclusion follows by Definition \ref{luoghibase}, we get \eqref{riduzgen}(b).

As for \eqref{unione2}, let $A$ be an ample Cartier divisor on $X$. By \eqref{riduzgen}(b) we can choose an $\varepsilon_0 > 0$ such that $\B_+(\alpha) = \B_{\rm num}(\alpha - \varepsilon [A^{n-k}])$, $\B_+(\beta) = \B_{\rm num}(\beta - \varepsilon [A^{n-k}])$ and $\B_+(\alpha+ \beta) = \B_{\rm num}(\alpha + \beta - \varepsilon [A^{n-k}])$ for every $0 < \varepsilon \leq \varepsilon_0$. Now \eqref{unione-ii} gives
\[ \B_+(\alpha + \beta) = \B_{\rm num}(\alpha + \beta - \varepsilon [A^{n-k}]) = \B_{\rm num}(\alpha - \frac{\varepsilon}{2} [A^{n-k}] + \beta - \frac{\varepsilon}{2} [A^{n-k}]) \subseteq \]
\[ \subseteq \B_{\rm num}(\alpha - \frac{\varepsilon}{2} [A^{n-k}]) \cup \B_{\rm num}(\beta - \frac{\varepsilon}{2} [A^{n-k}]) = \B_+(\alpha) \cup \B_+(\beta). \]

To see \eqref{numerabile}, for any ample $\R$-Cartier $\R$-divisors $A'_1, \ldots, A'_{n-k}$ we have, by \eqref{riduzgen}(a), that 
\[ \B_{\rm num}([A'_1 \cdots A'_{n-k}] -  \frac{1}{m} [A_1 \cdots A_{n-k}]) = \emptyset \ \mbox{for} \ m \gg 0. \]
But then, using \eqref{unione-ii},
\[ \B_{\rm num}(\alpha + [A'_1 \cdots A'_{n-k}]) =  \B_{\rm num}(\alpha + \frac{1}{m} [A_1 \cdots A_{n-k}] + [A'_1 \cdots A'_{n-k}] -  \frac{1}{m} [A_1 \cdots A_{n-k}]) \subseteq  \]
\[ \subseteq \B_{\rm num}(\alpha + \frac{1}{m} [A_1 \cdots A_{n-k}]) \cup  \B_{\rm num}([A'_1 \cdots A'_{n-k}] -  \frac{1}{m} [A_1 \cdots A_{n-k}]) = \B_{\rm num}(\alpha + \frac{1}{m} [A_1 \cdots A_{n-k}]) \]
and therefore
\[ \B_-(\alpha) \subseteq \bigcup\limits_{m \in \NN^+} \B_{\rm num}(\alpha + \frac{1}{m} [A_1 \cdots A_{n-k}]). \]
The other inclusion follows by definition of $\B_-(\alpha)$. Hence  \eqref{numerabile} is proved.

Finally to show \eqref{unb+} let $A'_1, \ldots, A'_{n-k}$ be ample $\R$-Cartier $\R$-divisors and let $A$ be an ample Cartier divisor on $X$. By \eqref{riduzgen} we can choose an $\varepsilon_0 > 0$ such that $\B_+(\alpha + [A'_1 \cdots A'_{n-k}]) = \B_{\rm num}(\alpha + [A'_1 \cdots A'_{n-k}] - \varepsilon [A^{n-k}])$ and $\B_{\rm num}([A'_1 \cdots A'_{n-k}] - \varepsilon [A^{n-k}]) = \emptyset$ for every $0 < \varepsilon \leq \varepsilon_0$. Now, using \eqref{unione-ii}, we get
\[ \B_+(\alpha + [A'_1 \cdots A'_{n-k}]) =  \B_{\rm num}(\alpha + [A'_1 \cdots A'_{n-k}] - \frac{\varepsilon}{2} [A^{n-k}]) = \]
\[ = \B_{\rm num}(\alpha + \frac{\varepsilon}{2} [A^{n-k}] + [A'_1 \cdots A'_{n-k}] - \varepsilon [A^{n-k}]) \subseteq \]
\[ \subseteq \B_{\rm num}(\alpha + \frac{\varepsilon}{2} [A^{n-k}]) \cup  \B_{\rm num}([A'_1 \cdots A'_{n-k}] - \varepsilon [A^{n-k}]) = \B_{\rm num}(\alpha + \frac{\varepsilon}{2} [A^{n-k}]) \subseteq \B_-(\alpha) \]
and therefore
\[ \bigcup\limits_{A'_1, \ldots, A'_{n-k}} \B_+(\alpha + [A'_1 \cdots A'_{n-k}]) \subseteq \B_-(\alpha). \]
The other inclusion follows by definition of $\B_-(\alpha)$ and \eqref{nest}.
\end{proof}

We have the following two consequences, the first one being Theorem \ref{b+}.

\begin{proof}[Proof of Theorem \ref{b+}]

\hskip 3cm

If $\B_+(\alpha) \subsetneq X$ then there exist ample $\R$-Cartier $\R$-divisors $A_1, \ldots, A_{n-k}$ on $X$ such that $\B_{\rm num}(\alpha - [A_1 \cdots A_{n-k}]) \subsetneq X$, so that $|\alpha - [A_1 \cdots A_{n-k}]|_{\rm num} \neq \emptyset$. Pick $e \in Z_k(X)_{\R}$ such that $e$ is effective and $[e] = \alpha - [A_1 \cdots A_{n-k}]$.  
Then, as in \cite[Lemma 2.12]{fl2}, $[A_1 \cdots A_{n-k}]$ is big, whence so is $\alpha = [A_1 \cdots A_{n-k}] + [e]$. Now assume that $\alpha$ is big. Then, given an ample Cartier divisor $A$ on $X$, there is an $\varepsilon > 0$ such that $\alpha - \varepsilon [A^{n-k}] \in \Eff_k(X)$, whence there is $e \in Z_k(X)_{\R}$ such that $e$ is effective and $\alpha = \varepsilon [A^{n-k}] + [e]$. Therefore
\[ \B_+(\alpha) \subseteq \B_{\rm num}(\alpha - [(\varepsilon^{\frac{1}{n-k}} A)^{n-k}]) \subseteq \Supp(e) \subsetneq X \]
and (i) follows. 
 
To see (ii) observe that, if $\alpha \in P_k(X)$, then there exist $A_1, \ldots, A_{n-k}$ ample $\R$-Cartier $\R$-divisors on $X$ and $\beta \in N_k(X)$ such that $\B_{\rm num} (\beta) = \emptyset$ and $\alpha = [A_1 \cdots A_{n-k}] + \beta$. Therefore $\B_+(\alpha) \subseteq \B_{\rm num}(\alpha - [A_1 \cdots A_{n-k}]) = \B_{\rm num} (\beta) = \emptyset$. On the other hand, assume that $\B_+(\alpha) = \emptyset$ and let $A$ be an ample Cartier divisor on $X$. By Proposition \ref{tutte}\eqref{riduzgen}(b) there is an 
$\varepsilon > 0$ such that $\B_+(\alpha) = \B_{\rm num}(\alpha - \varepsilon [A^{n-k}])$. Set $\beta = \alpha - \varepsilon [A^{n-k}]$. Then $\B_{\rm num}(\beta) = \emptyset$ and therefore $\alpha = [(\varepsilon^{\frac{1}{n-k}} A)^{n-k}] + \beta \in P_k(X)$.
\end{proof}

Using Theorem \ref{b+} we can make the union in Proposition \ref{tutte}\eqref{unb+} a countable one.

\begin{prop}
\label{unb+num}
Let $\alpha \in N_k(X)$ and let $A_1, \ldots, A_{n-k}$ be ample $\R$-Cartier $\R$-divisors on $X$. Then 
\[ \B_-(\alpha) = \bigcup\limits_{m \in \NN^+} \B_+(\alpha + \frac{1}{m} [A_1 \cdots A_{n-k}]). \]
\end{prop}
\begin{proof}
Let $A'_1, \ldots, A'_{n-k}$ be any ample $\R$-Cartier $\R$-divisors on $X$. By Proposition \ref{tutte}\eqref{riduzgen}(b) we have that 
\[ \B_{\rm num}([A'_1 \cdots A'_{n-k}] -  \frac{1}{m'} [A_1 \cdots A_{n-k}]) = \emptyset \ \mbox{for} \ m' \gg 0. \]
whence
\[[A'_1 \cdots A'_{n-k}] - \frac{1}{2m'} [A_1 \cdots A_{n-k}] =  \]
\[= [(\frac{1}{(2m')^{\frac{1}{n-k}}} A_1) \cdots (\frac{1}{(2m')^{\frac{1}{n-k}}} A_{n-k})] +[A'_1 \cdots A'_{n-k}] - \frac{1}{m'} [A_1 \cdots A_{n-k}]  \in P_k(X) \]
and therefore 
\[\B_+([A'_1 \cdots A'_{n-k}] - \frac{1}{2m'} [A_1 \cdots A_{n-k}]) = \emptyset \]
by Theorem \ref{b+}(ii). Then Proposition \ref{tutte}\eqref{unione2} gives
\[ \B_+(\alpha + [A'_1 \cdots A'_{n-k}]) \subseteq \B_+(\alpha + \frac{1}{2m'} [A_1 \cdots A_{n-k}]) \cup \B_+([A'_1 \cdots A'_{n-k}] - \frac{1}{2m'} [A_1 \cdots A_{n-k}]) \]
\[ = \B_+(\alpha + \frac{1}{2m'} [A_1 \cdots A_{n-k}]) \]
whence, by Proposition \ref{tutte}\eqref{unb+}, we get $\B_-(\alpha) \subseteq \bigcup\limits_{m \in \NN^+} \B_+(\alpha + \frac{1}{m} [A_1 \cdots A_{n-k}])$. The other inclusion follows again by Proposition \ref{tutte}\eqref{unb+}.
\end{proof}

\section{Push-forward  of complete intersections} 
\label{pf}

In this section we will prove that to compute stable base loci we can add more flexibility and use, instead of perturbations by complete intersections of ample divisors, perturbations by push-forward, under finite flat maps, of complete intersections of ample divisors. It was proved in \cite[Cor.~2.5.2]{ng} that a $k$-cycle $Z$ is numerically trivial if and only if $A_1 \cdots A_{k+e} \cdot f^*Z=0$ for all projective flat maps $f : Y \to X$ of relative dimension $e$ and for all $A_1, \ldots, A_{k+e}$ ample divisors on $Y$. By taking general hyperplane sections, we can reduce to the case where $f$ is finite flat. 

We introduce and study the corresponding cones.

\begin{defi}
We denote by $CI_k(X)$ the convex cone  generated by all classes $[A_1 \cdots A_{n-k}]$, where  $A_1, \ldots, A_{n-k}$ are ample $\R$-Cartier $\R$-divisors on $X$.

We denote by $PCI_k(X)$ the convex cone  generated by all classes of form $f_* [A_1 \cdots A_{n-k}]$, where $f : Y \to X$ runs among all finite flat morphisms and $A_1, \ldots, A_{n-k}$ run among all ample $\R$-Cartier $\R$-divisors on $Y$. 
\end{defi}

Given $\alpha \in N_k(X)$ it is easy to see, using the same proof of Proposition \ref{tutte}\eqref{riduzgen}(a)-(b),
that
\[ \B_+(\alpha)= \bigcap\limits_{\gamma \in CI_k(X)} \B_{\rm num}(\alpha - \gamma) \ \hbox{and} \ \B_-(\alpha)= \bigcup\limits_{\gamma \in CI_k(X)} \B_{\rm num}(\alpha + \gamma). \]
We now want to show that the same can be done for $PCI_k(X)$ and that, in fact, we get the same stable base loci.

We set (temporarily) 
\[ \B^{\rm pci}_+(\alpha)= \bigcap\limits_{f, A_1, \ldots, A_{n-k}} \B_{\rm num}(\alpha - f_* [A_1 \cdots A_{n-k}]) \]
and
\[ \B^{\rm pci}_-(\alpha)= \bigcup\limits_{f, A_1, \ldots, A_{n-k}} \B_{\rm num}(\alpha + f_* [A_1 \cdots A_{n-k}]) \]

where $f$ runs among all finite flat morphisms $f : Y \to X$ and $A_1, \ldots, A_{n-k}$ run among all ample $\R$-Cartier $\R$-divisors on $Y$. 

We have

\begin{lemma}
\label{pfo2}
Let $\alpha \in N_k(X)$. Let $f : Y \to X$ be a finite flat morphism and let $A_1, \ldots, A_{n-k}$ be ample $\R$-Cartier $\R$-divisors on $Y$. Then 
\begin{itemize}
\item[(i)] $\B_{\rm num}(f_* [A_1 \cdots A_{n-k}]) =  \emptyset.$
\end{itemize}
Moreover for any finite flat morphism $f' : Y' \to X$ and for any ample $\R$-Cartier $\R$-divisors $A'_1, \ldots, A'_{n-k}$ on $Y'$, there exists an $\varepsilon_0 > 0$ such that for every $0 < \varepsilon \leq \varepsilon_0$ we have
\begin{itemize}
\item[(ii)] $f_* [A_1 \cdots A_{n-k}] - \varepsilon f'_* [A'_1 \cdots A'_{n-k}] \in PCI_k(X)$; 
\item[(iii)] $\B_{\rm num}(f_* [A_1 \cdots A_{n-k}]) - \varepsilon f'_* [A'_1 \cdots A'_{n-k}]) = \emptyset$; 
\item[(iv)] $\B^{\rm pci}_+(\alpha)=\B_{\rm num}(\alpha - \varepsilon f'_* [A'_1 \cdots A'_{n-k}]).$
\end{itemize}
\end{lemma}
\begin{proof}
By Proposition \ref{tutte}\eqref{unione-i}-\eqref{unione-ii} it is enough to prove (i) when $A_1, \ldots, A_{n-k}$ are ample divisors on $Y$. Now let $x \in X$, so that that $f^{-1}(x)$ is a finite set and we can find effective divisors $E_i \sim_{\mathbb Q} A_i$ such that $E_i \cap f^{-1}(x) =  \emptyset$. Then for some $d>0$ we have that
$$f_* [A_1 \cdots A_{n-k}] = d [f(E_1 \cap \ldots \cap E_{n-k})]$$
and of course $x \not\in f(E_1 \cap \ldots \cap E_{n-k})$. This proves (i).

To see (ii) consider the following commutative diagram
$$\xymatrix{Y \times_X Y' \ar[d]^{\pi} \ar[dr]^{g} \ar[r]^{\ \ \ \pi'} & Y' \ar[d]^{f'} \\ Y \ar[r]^{f} & X}$$
where $g:= f \circ \pi = f' \circ \pi'$. By \cite[Prop.~III.9.2(b)]{h} it follows that $\pi$ and $\pi'$ are flat and by \cite[Prop.~6.1.5(iii)]{g} they are finite, whence also $g$ is finite by \cite[Lemma 29.43.5]{st} and flat by \cite[Prop.~III.9.2(c)]{h}. Then $\pi^*A_i$ and $(\pi')^*A'_i$ are ample for all $1 \le i \le n-k$. As in the proof of Proposition \ref{tutte}\eqref{riduzgen}(a), there exists an $\varepsilon_1 > 0$ such that there are ample $\R$-Cartier $\R$-divisors $A''_{ij}$ on $Y \times_X Y'$ satisfying
$$[\pi^*A_1 \cdots \pi^* A_{n-k}] - \varepsilon_1 [(\pi')^*A'_1 \cdots (\pi')^*A'_{n-k}] = \sum\limits_j [A''_{1j} \cdots A''_{n-k,j}].$$
Hence for some $d>0, d'>0$ we have that
$$f_* [A_1 \cdots A_{n-k}] - \frac{\varepsilon_1 d'}{d} f'_* [A'_1 \cdots A'_{n-k}] = \frac{1}{d} g_*[\pi^*A_1 \cdots \pi^* A_{n-k}] - \frac{\varepsilon_1}{d}g_* [(\pi')^*A'_1 \cdots (\pi')^*A'_{n-k}] =$$
$$= \sum\limits_j \frac{1}{d} g_*[A''_{1j} \cdots A''_{n-k,j}]$$
and then, setting $\varepsilon_0 = \frac{\varepsilon_1 d'}{d}$, (ii) follows and so does (iii) by (i) and Proposition \ref{tutte}\eqref{unione-i}-\eqref{unione-ii}.
Finally to show (iv) choose finite flat morphisms $f_i : Y_i \to X$ and ample $\R$-Cartier $\R$-divisors $A_{i1},  \ldots, A_{i, n-k}$ on $Y_i$, $1 \leq i \leq s$, such that 
 \[ \B^{\rm pci}_+(\alpha)= \bigcap\limits_{i=1}^s \B_{\rm num}(\alpha - (f_i)_*[A_{i1} \cdots A_{i, n-k}]). \]
 By (ii) we can choose $\varepsilon > 0$ such that  
 \[ \B_{\rm num} ((f_i)_* [A_{i1} \cdots A_{i,n-k}] - \varepsilon f'_*[A'_1 \cdots A'_{n-k}]) = \emptyset \ \mbox{for all} \ 1 \leq i \leq s. \]
 Therefore, using Proposition \ref{tutte}\eqref{unione-ii}, for all $1 \leq i \leq s$ we have
\[ \B_{\rm num}(\alpha - \varepsilon f'_*[A'_1 \cdots A'_{n-k}]) \subseteq \B_{\rm num}(\alpha - (f_i)_*[A_{i1} \cdots A_{i, n-k}]) \cup \B_{\rm num}((f_i)_*[A_{i1} \cdots A_{i, n-k}] - \varepsilon f'_*[A'_1 \cdots A'_{n-k}]) =$$
$$=  \B_{\rm num}(\alpha - (f_i)_*[A_{i1} \cdots A_{i, n-k}])\]
hence $\B_{\rm num}(\alpha - \varepsilon f'_*[A'_1 \cdots A'_{n-k}]) \subseteq \B^{\rm pci}_+(\alpha)$. Since the other inclusion follows by definition of $\B^{\rm pci}_+(\alpha)$, we get (iv).
\end{proof}

\begin{prop}
\label{pfo}
Let $\alpha \in N_k(X)$. Then
\[ \B_+(\alpha) = \bigcap\limits_{f, A_1, \ldots, A_{n-k}} \B_{\rm num}(\alpha - f_* [A_1 \cdots A_{n-k}]) \]

\[ \B_-(\alpha)= \bigcup\limits_{f, A_1, \ldots, A_{n-k}} \B_{\rm num}(\alpha + f_*[A_1 \cdots A_{n-k}]) \]
where $f$ runs among all finite flat morphisms $f : Y \to X$ and $A_1, \ldots, A_{n-k}$ run among all ample $\R$-Cartier $\R$-divisors on $Y$. 
\end{prop}
\begin{proof}
By Proposition \ref{tutte}\eqref{riduzgen}(b), there exists an $\varepsilon_0 > 0$ such that for all $0 < \varepsilon \le \varepsilon_0$ we have that $\B_+(\alpha) = \B_{\rm num}(\alpha - \varepsilon [A_1 \cdots A_{n-k}])$. By Lemma \ref{pfo2}(iii) there exists an $\varepsilon_1 > 0$ such that for all $0 < \varepsilon \le \varepsilon_1$ we have that $\B^{\rm pci}_+(\alpha)=\B_{\rm num}(\alpha - \varepsilon [A_1 \cdots A_{n-k}])$. Choosing $\varepsilon = \min\{\varepsilon_0, \varepsilon_1\}$ we conclude that $\B_+(\alpha)=\B^{\rm pci}_+(\alpha)$.

As for $\B_-(\alpha)$, by definition it follows that $\B_-(\alpha) \subseteq \B^{\rm pci}_-(\alpha)$. Let now $A$ be an ample divisor on $X$. Given any finite flat morphism $f : Y \to X$ and any ample $\R$-Cartier $\R$-divisors $A_1, \ldots, A_{n-k}$ on $Y$, we have by Lemma \ref{pfo2}(ii) that there exists an $\varepsilon_0 > 0$ such that   if $0 < \varepsilon \leq \frac{\varepsilon_0}{2}$ we have that
\begin{equation}
\label{zz}
\B_{\rm num}(f_* [A_1 \cdots A_{n-k}]) - 2\varepsilon [A^{n-k}]) = \emptyset.
\end{equation}

Moreover by Proposition \ref{tutte}\eqref{riduzgen}(b) there exists an $\varepsilon_1 > 0$ such that for every $0 < \varepsilon \leq \varepsilon_1$ we have that
\begin{equation}
\label{zz2}
\B_+(\alpha + f_*[A_1 \cdots A_{n-k}]) = \B_{\rm num}(\alpha + f_*[A_1 \cdots A_{n-k}] - \varepsilon [A^{n-k}]).
\end{equation}
Then for $\varepsilon = \min\{\frac{\varepsilon_0}{2},\varepsilon_1\}$ we have, using Proposition \ref{tutte}\eqref{nest}, \eqref{zz2}, Proposition \ref{tutte}\eqref{unione-ii} and \eqref{zz}, that
\[\B_{\rm num}(\alpha + f_*[A_1 \cdots A_{n-k}]) \subseteq \B_+(\alpha + f_*[A_1 \cdots A_{n-k}])=\B_{\rm num}(\alpha + f_*[A_1 \cdots A_{n-k}] - \varepsilon [A^{n-k}])=\]
\[= \B_{\rm num}(\alpha + \varepsilon [A^{n-k}] + f_*[A_1 \cdots A_{n-k}] - 2\varepsilon [A^{n-k}]) \subseteq \B_{\rm num}(\alpha + \varepsilon [A^{n-k}]) \subseteq \B_-(\alpha).\]
Therefore also $\B^{\rm pci}_-(\alpha) \subseteq \B_-(\alpha)$ and we are done.
\end{proof}

Again, using Proposition \ref{tutte}\eqref{unione-ii},\eqref{riduzgen}(b), Lemma \ref{pfo2} and Proposition \ref{pfo}, it follows easily that

\begin{remark} Let $\alpha \in N_k(X)$. Then
\label{pert}
\[ \B_+(\alpha)= \bigcap\limits_{\gamma \in PCI_k(X)} \B_{\rm num}(\alpha - \gamma). \]
\[ \B_-(\alpha)= \bigcup\limits_{\gamma \in PCI_k(X)} \B_{\rm num}(\alpha + \gamma). \]
\end{remark}

As for the cones we get

\begin{prop}
\label{coni3}
We have
\[P_k(X)=CI_k(X)+NSA_k(X)=PCI_k(X)+NSA_k(X).\]
\end{prop}
\begin{proof}
The first equality follows by definition of $P_k(X)$ and convexity of $NSA_k(X)$. To see the second we first show that 
\begin{equation}
\label{cont}
PCI_k(X) \subseteq P_k(X).
\end{equation}
By Theorem \ref{b+}(ii) it is enough to prove that $\B_+(\alpha) = \emptyset$ for every $\alpha \in PCI_k(X)$. On the other hand, by Proposition \ref{tutte}\eqref{unione2}, it is sufficient to show that $\B_+(\alpha) = \emptyset$ when $\alpha = f_* [A_1 \cdots A_{n-k}]$, where $f : Y \to X$ is a finite flat morphism and $A_1, \ldots, A_{n-k}$ are ample $\R$-Cartier $\R$-divisors on $Y$. By Proposition \ref{pfo} and Lemma \ref{pfo2}(iii) there exists an  $\varepsilon$ such that $0 < \varepsilon < 1$ and $\B_+(\alpha) = \B_{\rm num}(\alpha - \varepsilon f_*[A_1 \cdots A_{n-k}]) = \B_{\rm num}((1-\varepsilon) f_*[A_1 \cdots A_{n-k}])$ and the latter is empty by Proposition \ref{tutte}\eqref{unione-i} and Lemma \ref{pfo2}(i). This proves \eqref{cont}.

As $NSA_k(X)$ is convex, we get by \eqref{cont} that
\[ PCI_k(X)+NSA_k(X) \subseteq P_k(X) + NSA_k(X) = CI_k(X)+NSA_k(X)  \subseteq PCI_k(X)+NSA_k(X)\]
and we are done.
\end{proof}

On a given variety one can consider various cones of positive cycles. For example, as in Lemma \ref{vb} below, one can consider the convex cones generated by Chern classes, or dual Segre classes or Schur classes or even monomials in Schur classes of several sufficiently positive vector bundles. Aside from what we know from Lemma \ref{vb}, we can observe that the cone of dual Segre classes of ample vector bundles is a subcone of $PCI_k(X)$. As a matter of fact, any dual Segre class $s_{n-k}(\E^\vee)$ of an ample vector bundle $\E$ is push-forward, from $P(\E)$ of a power of the ample line bundle $\O_{P(\E)}(1)$. By cutting down with hyperplanes one obtains that $s_{n-k}(\E^\vee) \in PCI_k(X)$.

We do not know if this holds more generally for other types of Schur classes.

\section{Proof of Theorem \ref{aperto}} 
\label{pfthm1}

We will give two different proofs of Theorem \ref{aperto}. The first one holds for every $k$ and uses the cone $PCI_k(X)$. The other one holds only for $k < \frac{n}{2}+1$ or $k=n-1$, but it has the advantage to introduce a method of studying cycles that makes them resemble divisors. This method might be useful in the study of higher codimensional cycles. 

Here is the first proof of Theorem \ref{aperto}.

\begin{proof}
By Proposition \ref{tutte}\eqref{unione-i}-\eqref{unione-ii} we see that $PCI_k(X)$ is a convex cone. Now assume that $X$ is smooth. It is easily seen that any convex cone satisfying a property such as the one in Lemma \ref{pfo2}(ii), coincides with its relative interior. Therefore $PCI_k(X)$ is open in its linear span in $N_k(X)$. On the other hand, as mentioned in the beginning of section \ref{pf}, it follows by [Ng, Cor. 2.5.2] that $PCI_k(X)$ generates $N_{n-k}(X)^{\vee} = N_k(X)$. By Proposition \ref{coni3} we have, by a simple fact of convex cones,
that
\[P_k(X)=PCI_k(X)+NSA_k(X) = \Int(NSA_k(X))\]
and therefore $P_k(X)$ is open in $N_k(X)$, hence also full-dimensional.
\end{proof}

To give the second proof, we first need three lemmas.

\begin{lemma}
\label{vb}
Let $\E$ be a vector bundle of rank $r$ and let $A$ be an ample Cartier divisor on $X$.
\begin{itemize}
\item[(i)] If $\E$ is globally generated, then $\B_{\rm num}([c_{n-k}(\E)]) = \emptyset$; 
\item[(ii)] If $\E(-A)$ is globally generated and $r \ge n-k$, then $[c_{n-k}(\E)] \in P_k(X)$;
\item[(iii)] Assertions $(i)$ and $(ii)$ hold more generally for Schur classes of $\E$ if $X$ is smooth and we are in characteristic zero.
\end{itemize}
\end{lemma}
\begin{proof}
To see (i), the assertion being obvious if $r < n-k$, assume that $r \ge n-k$. Let $x \in X$ and pick general sections $\tau_0, \ldots, \tau_{r - n + k} \in H^0(\E)$. Then they are linearly independent in $x$ and therefore $x$ does not belong to their degeneracy locus, which, as is well-known  \cite[Examples 14.3.2 and 14.4.3]{fu}, \cite[Lemma 5.2]{eh}, represents $[c_{n-k}(\E)]$. This proves (i). To see (ii) observe that we can write
\[ [c_{n-k}(\E)] = \binom{r}{n-k}[A^{n-k}] + \sum\limits_{j=0}^{n-k-1} \binom{r-n+k+j}{j}[A^j c_{n-k-j}(\E(-A))] \]
and, using Proposition \ref{tutte}\eqref{unione-i}-\eqref{unione-ii}, we see that this belongs to $P_k(X)$ since (i) implies that
\[\B_{\rm num}([A^j c_{n-k-j}(\E(-A))]) = \emptyset \ \mbox{for all} \ 0 \le j \le n-k-1. \] 
This gives (ii). Let now $\lambda = (\lambda_1,\ldots,\lambda_{n-k})$ be a partition of $n-k$ with $r \ge \lambda_1 \ge \ldots  \ge \lambda_{n-k} \ge 0$. As for the first part of (iii) we just notice that, when $\E$ is globally generated it defines a morphism $f : X \to \mathbb G$ to a Grassmannian and then any Schur class $s_{\lambda}(\E)$ is the pull-back $f^*s_{\lambda}(Q)$, where $Q$ is the tatutological quotient bundle. By \cite[Rmk.~8.3.6]{l}, there is a Schubert variety $\Omega_{\lambda}$ such that $[s_{\lambda}(Q)] = [\Omega_{\lambda}]$. Given any $x \in X$, since $\mathbb G$ is a homogeneous space, a general translate $g\Omega_{\lambda}$ of $\Omega_{\lambda}$ does not contain $f(x)$, is generically transverse to $f$ and $[g\Omega_{\lambda}]=[\Omega_{\lambda}]$ by Kleiman’s transversality theorem \cite[Thm.~1.7]{eh}. By \cite[Thm.~1.23]{eh} we have that $[f^{-1}(g\Omega_{\lambda})] = f^*[g\Omega_{\lambda}]=f^*[\Omega_{\lambda}]=f^*[s_{\lambda}(Q)] = [s_{\lambda}(\E)]$ and of course $x \not\in f^{-1}(g\Omega_{\lambda})$. Hence (i) follows for $s_{\lambda}(\E)$. When $\E(-A)$ is globally generated, the proof of (ii) for Schur classes is similar to the proof of (ii) for Chern classes. In fact by \cite[Cor.~7.2]{p} (or \cite[\S 2.5]{rt}) we have that
\[ [s_{\lambda}(\E)] = d^{\lambda}_{(1,\ldots,1)} [A^{n-k}] + \sum\limits_{\nu} d^{\lambda}_{\nu}[A^{n-k-|\nu|}s_{\nu}(\E(-A))] \]
and $d^{\lambda}_{(1,\ldots,1)} > 0$, $d^{\lambda}_{\nu} \ge 0$ for all $\nu$. Then the proof goes as above.
\end{proof}

\begin{lemma}
\label{base}
Assume that $X$ is smooth and that $k < \frac{n}{2}+1$. 
Then there are $V_1, \ldots, V_p$ smooth subvarieties of $X$ of dimension $k$ such that $\{[V_1], \ldots, [V_p]\}$ is a basis of $N_k(X)$.
\end{lemma}
\begin{proof}
This follows by \cite[Thm.~5.8]{k}.
\end{proof}

\begin{remark}
As it will be clear from the sequel, it would actually be enough to have a basis of classes of lci subvarieties. It is an open problem that such a basis exists for every $k$. In \cite[Conj.~2.1]{mv} it is conjectured to hold for rational equivalence, whence also for numerical equivalence.
\end{remark}

\begin{lemma}
\label{basesing}
Assume that $X$ is smooth and let $V$ be a smooth subvariety of $X$ of dimension $k$ with $1 \le k \le n-2$. Let $A$ be an ample Cartier divisor on $X$ and let $m \gg 0$. For every $x \in X$ there exist divisors $D_i \in |\I_{V/X}(m A)|,  1 \leq i \leq n-k-1$, such that if $Y$ is the complete intersection of $D_1, \ldots,  D_{n-k-1}$, then $x \not\in \Sing(Y)$. 
\end{lemma}
\begin{proof}
Let $m \gg 0$ be such that $\I_{V/X}(mA)$ is globally generated and $H^1(\I_{V/X}^2 (mA))=0$. If $x \not\in V$, then for a general choice of $D_1, \ldots,  D_{n-k-1}$, we actually have that $x \not\in Y$. Now suppose that $x \in V$. Since $(\I_{V/X} / \I_{V/X}^2) (mA)$ is also globally generated and has rank $n-k$, there are $\sigma_1, \ldots, \sigma_{n-k-1} \in H^0((\I_{V/X} / \I_{V/X}^2) (mA))$ such that they are linearly independent in $x$. From the exact sequence
$$0 \to \I_{V/X}^2 (mA) \to \I_{V/X} (mA) \to (\I_{V/X} / \I_{V/X}^2) (mA)$$
we get $f_1, \ldots, f_{n-k-1} \in H^0(\I_{V/X}(m A))$ such that $\sigma_i = df_i, 1 \le i \le n-k-1$. Hence $df_1, \ldots, df_{n-k-1}$ are linearly independent in $x$, that is $x \not\in \Sing(Y)$, where $Y$ is the complete intersection of the divisors $D_1, \ldots,  D_{n-k-1}$ associated to $f_1, \ldots, f_{n-k-1}$. 
\end{proof}

We now proceed to give the second proof, holding only when $k < \frac{n}{2}+1$  or $k=n-1$, of Theorem \ref{aperto}.

\begin{proof}
By Proposition \ref{tutte}\eqref{unione-i}-\eqref{unione-ii} we see that $P_k(X)$ is a convex cone. Now assume that $X$ is smooth. Let $p = \dim N_k(X)$. It follows by \cite[Cor.~2.5.2]{j} (see also \cite[Rmk.~2.2]{fl2}) that there are vector bundles $\E_1, \ldots, \E_p$ on $X$ such that $\{[c_{n-k}(\E_1)], \ldots, [c_{n-k}(\E_p)]\}$ is a basis of $N_k(X)$. In particular $\rk \E_j \ge n-k$ for all $1 \le j \le p$. Let $A$ be an ample Cartier divisor and let $m_0$ be such that $\E_j(mA)$ is globally generated for all $1 \le j \le p$ and for all $m \ge m_0$. Now let $\{\phi_1, \ldots, \phi_p\}$ be a basis of $N_k(X)^{\vee}$. Then the matrix $(\phi_i([c_{n-k}(\E_j)]))$ is nondegenerate, whence so is the matrix $(\phi_i([c_{n-k}(\E_j(mA))]))$ for $m \gg 0$, because its determinant is a non-zero polynomial in $m$. Therefore $\{[c_{n-k}(\E_1(mA))], \ldots, [c_{n-k}(\E_p(mA))]\}$ is a basis of $N_k(X)$.  On the other hand, $[c_{n-k}(\E_j(mA))] \in P_k(X)$ for all $1 \le j \le p$ by Lemma \ref{vb}(ii) and therefore $P_k(X)$ is full-dimensional in $N_k(X)$.

If $k=n-1$ it folllows by Lemma \ref{divis}(i) that $P_{n-1}(X) = \Amp(X)$, whence it is open. Suppose next that $k < \frac{n}{2}+1$ and $k \le n-2$. 

By Lemma \ref{base} there are $V_1, \ldots, V_p$ smooth subvarieties of $X$ of dimension $k$ such that $\{[V_1], \ldots, [V_p]\}$ is a basis of $N_k(X)$. To see that $P_k(X)$  is open it is enough to prove that, if $\alpha \in P_k(X)$, then there is a $\delta > 0$ such that 
\begin{equation}
\label{open}
\alpha + \varepsilon_1 [V_1] + \ldots + \varepsilon_p [V_p] \in P_k(X) \ \mbox{for all} \ \varepsilon_i \in \R \ \mbox{such that} \ |\varepsilon_i| < \delta, 1 \leq i \leq p.
\end{equation}
Let $A_1, \ldots, A_{n-k}$ be ample $\R$-Cartier $\R$-divisors on $X$ and let $\beta \in N_k(X)$ such that $\B_{\rm num} (\beta) = \emptyset$ and
\begin{equation}
\label{uno}
\alpha = [A_1 \cdots A_{n-k}] + \beta.
\end{equation}
For $j = 1, \ldots, n-k$ we can write $A_j = c_j A'_j + A''_j$ with $c_j \in \R^+, A'_j$ ample Cartier divisor and $A''_j$ zero or ample $\R$-Cartier $\R$-divisor. Let $A$ be a very ample Cartier divisor and let $s_j$ be such that $s_jA'_j-A$ is ample. Then we can write $A_j = c'_j A + A'''_j$ with $c'_j \in \R^+$ and $A'''_j$ ample $\R$-Cartier $\R$-divisor. Setting $c = c'_1 \cdots c'_{n-k}$, we have
\begin{equation}
\label{due}
[A_1 \cdots A_{n-k}] = c [A^{n-k}] + \gamma
\end{equation}
where $\gamma \in N_k(X)$ is a class that is either zero or sum of intersections of $n-k$ ample $\R$-Cartier $\R$-divisors. In particular $\B_{\rm num} (\gamma) = \emptyset$. Let $m_0 \gg 0$ be such that Lemma \ref{basesing} holds for all $V_i, 1 \leq i \leq p$. Let $D_{ij} \in |\I_{V_i/X}(m_0 A)|$ be general divisors and let $Y_i$ be the complete intersection of $D_{i,1}, \ldots,  D_{i,n-k-1}$. Let $\O_{Y_i}(\pm V_i)$ be the sheaf associated to the Weil divisor $\pm V_i$. Now let $m_1 \gg 0$ be such that $H^q(Y_i, \O_{Y_i}(\pm V_i)((m_1 - q)A)) = 0$ for every $q > 0$ and for all $1 \leq i \leq p$. Note that then $\O_{Y_i}(\pm V_i)(m_1 A)$ is $0$-regular, whence globally generated for all $1 \leq i \leq p$. Set $m = m_0^{n-k-1} m_1$. Then there are effective $k$-cycles $e_i, f_i$ on $X$ such that, for all $1 \leq i \leq p$, we have
\begin{equation}
\label{bp}
m [A^{n-k}] = [V_i] + [e_i] \ \mbox{and} \ m [A^{n-k}] = - [V_i] + [f_i].
\end{equation}
We claim that, for all $1 \leq i \leq p$, we have $\B_{\rm num} ([e_i]) = \B_{\rm num} ([f_i]) = \emptyset$. 

In fact let $x \in X$. By Lemma \ref{basesing} we have that for a general choice of divisors $D'_{ij} \in |\I_{V_i/X}(m_0 A)|,  1 \leq i \leq n-k-1$, we have that if $Y'_i$ is their complete intersection, then $x \not\in \Sing(Y'_i)$. By semicontinuity we have that $H^q(Y'_i, \O_{Y'_i}(\pm V_i)((m_1 - q)A)) = 0$ for every $q > 0$ and for all $1 \leq i \leq p$ and therefore $\O_{Y'_i}(\pm V_i)(m_1 A)$ is $0$-regular, whence globally generated for all $1 \leq i \leq p$. Then both $\O_{Y'_i}(-V_i)(m_1 A)$ and $\O_{Y'_i}(V_i)(m_1 A)$ are globally generated line bundles in a neighborhood of $x$ and therefore we can find $e'_i, f'_i$, of the same class  $m [A^{n-k}] \pm [V_i]$ on $X$ of $e_i, f_i$, and such that $x \not\in \Supp(e'_i) \cup \Supp(f'_i)$.

Now let 
\[ \delta = \frac{c}{4pm} \]
and assume that $|\varepsilon_i| < \delta$ for all $1 \leq i \leq p$. Let $s := \# \{i \in \{1, \dots, p \} : \varepsilon_i < 0 \}$ and set $A' = (\frac{c}{2})^{\frac{1}{n-k}}A$. Then, using \eqref{uno} and \eqref{due}, we can write
\[ \alpha + \sum\limits_{i=1}^p \varepsilon_i [V_i] = c [A^{n-k}] +  \beta + \gamma + \sum\limits_{i=1}^p \varepsilon_i [V_i] = \]
\[ = [(A')^{n-k}] + \beta + \gamma + \frac{c}{2} [A^{n-k}] + \sum\limits_{i=1}^p \varepsilon_i [V_i]  \]
and now
\[ \frac{c}{2} [A^{n-k}] + \sum\limits_{i=1}^p \varepsilon_i [V_i] = \] 
\[ \sum\limits_{1 \leq i \leq p: \varepsilon_i < 0} (\frac{c}{4s}  [A^{n-k}] + \varepsilon_i [V_i]) + \sum\limits_{1 \leq i \leq p: \varepsilon_i \geq 0} (\frac{c}{4(p-s)}  [A^{n-k}] + \varepsilon_i [V_i]) \]
where the first sum is empty if $s = 0$ and the second sum is empty if $s = p$. Since $[(A')^{n-k}]$ is intersection of $n-k$ ample $\R$-Cartier $\R$-divisors, and $\B_{\rm num} (\beta) = \B_{\rm num} (\gamma) = \emptyset$, by Proposition \ref{tutte}\eqref{unione-i}-\eqref{unione-ii} we see that \eqref{open} will be proved as soon as we show that
\[ \B_{\rm num} (\frac{c}{4s}  [A^{n-k}] + \varepsilon_i [V_i]) = \emptyset \ \mbox{for all} \ i \ \mbox{such that} \ \varepsilon_i  < 0 \]
and
\[ \B_{\rm num} (\frac{c}{4(p-s)}  [A^{n-k}] + \varepsilon_i [V_i]) = \emptyset \ \mbox{for all} \ i \ \mbox{such that} \ \varepsilon_i  \geq 0. \]
On the other hand, the latter clearly holds, again by Proposition \ref{tutte}\eqref{unione-i}-\eqref{unione-ii}, since, by \eqref{bp}, we can write
\[ \frac{c}{4s}  [A^{n-k}] + \varepsilon_i [V_i] = (\frac{c}{4s} + \varepsilon_i m) [A^{n-k}] - \varepsilon_i [e_i] \]
for all $i$ such that $\varepsilon_i  < 0$ and 
\[ \frac{c}{4(p-s)}  [A^{n-k}] + \varepsilon_i [V_i] = (\frac{c}{4(p-s)} - \varepsilon_i m) [A^{n-k}] + \varepsilon_i [f_i] \]
for all $i$ such that $\varepsilon_i  \geq 0$, observing that
\[ \frac{c}{4s} + \varepsilon_i m > 0 \ \mbox{and} \ \frac{c}{4(p-s)} - \varepsilon_i m > 0 \]
by the choice of $\delta$.
\end{proof}

\section{The cone of numerically semiample cycles} 
\label{nsam}

It is clear that if $\alpha \in P_k(X)$, then $\B_{\rm num} (\alpha) = \emptyset$. This allows to introduce a larger cone.

\begin{defi}
\label{nsa}
The {\it cone of numerically semiample cycles} is
\[ NSA_k(X) = \{\alpha \in N_k(X):  \B_{\rm num} (\alpha) = \emptyset \}. \]
\end{defi}

It follows by Proposition \ref{tutte}\eqref{unione-i}-\eqref{unione-ii} that $NSA_k(X)$ is a convex cone. 

The first consequence of Theorem \ref{aperto} is the following.

\begin{lemma}
\label{int}
We have $\Int(NSA_k(X)) \subseteq P_k(X) \subseteq NSA_k(X)$. Moreover, if $X$ is smooth, then $P_k(X)$ is the interior of $NSA_k(X)$ and $\overline{P_k(X)} = \overline{NSA_k(X)}$.
\end{lemma}
\begin{proof}
By Proposition \ref{tutte}\eqref{unione-ii} we have that $P_k(X) \subseteq NSA_k(X)$. Let $\alpha \in \Int(NSA_k(X))$ and let $A$ be an ample Cartier divisor. Then there exists an $\varepsilon > 0$ such that 
\[ \alpha - \varepsilon [A^{n-k}] \in NSA_k(X) \]
and setting $\beta = \alpha - \varepsilon [A^{n-k}]$ we get that $\B_{\rm num} (\beta) = \emptyset$ and
\[ \alpha = [(\varepsilon^{\frac{1}{n-k}} A)^{n-k}] + \beta \in P_k(X). \]
Now suppose that $X$ is smooth. It follows by Theorem \ref{aperto} that $P_k(X) = \Int(NSA_k(X))$. Finally $\overline{P_k(X)} = \overline{NSA_k(X)}$ holds because $NSA_k(X)$ is a convex cone.
\end{proof}

We record some general properties of these cones.

\begin{remark} 
\label{coni}
\hskip 3cm
\begin{itemize}
\item[(i)] $P_k(X)$ and $NSA_k(X)$ are convex salient cones; 
\item[(ii)] $NSA_k(X)$ is not, in general, neither open nor closed. 
\end{itemize}
\end{remark}
\begin{proof}
We have already seen, in Theorem \ref{aperto} and after definition \ref{nsa}, that $P_k(X)$ and $NSA_k(X)$ are convex cones. Now let $\alpha \in NSA_k(X)$ and assume that also $-\alpha \in NSA_k(X)$. Pick $e, f \in Z_k(X)_{\R}$ such that $e$ and $f$ are effective and $[e] = \alpha, [f] = -\alpha$ and let $A$ be an ample Cartier divisor on $X$. Then $\alpha \cdot [A^{n-k}] = [e] \cdot [A^{n-k}] \geq 0$ and $- \alpha \cdot [A^{n-k}] = [f] \cdot [A^{n-k}] \geq 0$, so that $\alpha \cdot [A^{n-k}] = 0$ and therefore $e = 0$. This gives (i).

To see (ii) take a nef non semiample divisor $D$ on a smooth surface $X$ with $q(X) = 0$, such as in Zariski's example \cite{z}, \cite[Example 2.3A]{l}. Then $[D] \in \overline{NSA_1(X)}$ but $\B_{\rm num} ([D]) = \B(D) \neq \emptyset$, so that $[D] \not\in NSA_1(X)$. Hence $NSA_1(X)$ is not closed. Now take a semiample non big divisor $D$ on some smooth surface $X$. Then $\B_{\rm num} ([D]) = \emptyset$ so that $[D] \in NSA_1(X)$, but $[D]$ is not in the interior of $NSA_1(X)$, for otherwise, by Lemma \ref{int}, $[D] \in P_1(X)$, so that $D$ is big. Therefore $NSA_1(X)$ is not open.
\end{proof}

\begin{lemma}
\label{coni2}
We have
\begin{itemize}
\item[(i)] $P_k(X) \subseteq {\rm Big}_k(X)$ and $\overline{P_k(X)} \subseteq \overline{\Eff_k(X)}$; 
\item[(ii)] Assume that $\Eff_k(X) \subseteq NSA_k(X)$ (for example if $X$ is an abelian or homogeneous variety). Then
$P_k(X) = {\rm Big}_k(X)$ and $\overline{P_k(X)} = \overline{\Eff_k(X)}$;
\item[(iii)] In general, if $X$ is smooth, $P_k(X) \not\subseteq \Nef_k(X)$.
\end{itemize}
\end{lemma}
\begin{proof}
To see (i) let $\alpha \in P_k(X)$. Then there exist ample $\R$-Cartier $\R$-divisors $A_1, \ldots, A_{n-k}$ and $\beta \in N_k(X)$ such that $\B_{\rm num} (\beta) = \emptyset$ and $\alpha = [A_1 \cdots A_{n-k}] + \beta$. Pick $e \in Z_k(X)_{\R}$ such that $e$ is effective and $[e] = \beta$. As in \cite[Lemma 2.12]{fl2} we have that $[A_1 \cdots A_{n-k}]$ is big, whence $\alpha \in \Int(\overline{\Eff_k(X)}) + \overline{\Eff_k(X)} \subseteq  \Int(\overline{\Eff_k(X)}) = {\rm Big}_k(X)$ and (i) follows. 

To see (ii) let $\alpha \in {\rm Big}_k(X)$. Then, given an ample Cartier divisor $A$ on $X$, there is an $\varepsilon > 0$ such that $\alpha - \varepsilon [A^{n-k}] \in \Eff_k(X)$, whence $\alpha \in P_k(X)$. This gives (ii).

Now let $X$ be the blow-up of a smooth variety of dimension $n \ge 3$ at a point and let $E \subset X$ be the exceptional divisor, so that $E \cong \PP^{n-1}$. Let $\beta \in N_1(X)$ be the class of a line in $E$ so that $\beta.E = -1$ and $\B_{\rm num} (\beta) = \emptyset$. Let $A$ be an ample Cartier divisor on $X$ and let $m \gg 0$ be such that $(\frac{1}{m} [A^{n-1}] + \beta).E < 0$. Then $\alpha = \frac{1}{m} [A^{n-1}] + \beta \in P_1(X) \setminus \Nef_1(X)$.
\end{proof} 
It is clear that, in general, the inclusions in Lemma \ref{coni2}(i) can be strict. For example if $X$ is smooth we have by Lemma \ref{divis}(i) that $P_{n-1}(X) = \Amp(X)$ whence we can have strict inclusions.

\begin{remark} 
\label{comp}
Using Grassmannians and the example in the proof of Lemma \ref{coni2}(iii) it is easy to see that, in general, $P_k(X)$ is not contained, neither contains the positive cones defined in \cite{fl2}. It would be nice to understand the relation between $P_k(X)$ and the cone generated by ample subschemes. Since ample lci subschemes are nef \cite[\S 4]{o1}, it follows that in general $P_k(X)$ is not contained in the cone generated by ample lci subschemes.
\end{remark}

\begin{remark} 
\label{zd}
For a cycle $\alpha \in \overline{\Eff_k(X)}$ there are two notions of decomposition, the $\sigma$-decomposition in \cite[III.2]{na} and the Zariski decomposition in \cite{fl3}. It would be nice to understand the relation between $\B_-(\alpha)$ and the negative part of the decomposition, possibly resembling the case of divisors.
\end{remark}

\section{Proof of Theorem \ref{b-}} 
\label{pfthm3}

\begin{proof}
Let $A$ be an ample Cartier divisor on $X$. If $\B_-(\alpha) \subsetneq X$, then, for every $m \in \NN^+$, we have that $\B_{\rm num}(\alpha + \frac{1}{m} [A^{n-k}]) \subsetneq X$, whence $|\alpha +  \frac{1}{m} [A^{n-k}]|_{\rm num} \neq \emptyset$. Pick $e_m \in Z_k(X)_{\R}$ such that $e_m$ is effective and $[e_m] = \alpha +  \frac{1}{m} [A^{n-k}]$. Then $\alpha = \lim\limits_{m \to \infty} e_m$ is pseudoeffective. This proves (i). 

To see (ii) assume that the base field is uncountable and that $\alpha$ is pseudoeffective. Then $\alpha + \frac{1}{m} [A^{n-k}]$ is big for every $m \in \NN^+$ and therefore $\B_{\rm num}(\alpha + \frac{1}{m} [A^{n-k}]) \subsetneq X$ by Theorem \ref{b+}(i). Then also
$\B_-(\alpha) \subsetneq X$ by Proposition \ref{tutte}\eqref{numerabile} and this gives (ii).

If $\B_-(\alpha) = \emptyset$, then, for every $m \in \NN^+$, we have $\B_{\rm num} (\alpha + \frac{1}{m} [A^{n-k}]) = \emptyset$. Set $\beta_m = \alpha + \frac{1}{m} [A^{n-k}]$. Then $\alpha = \lim\limits_{m \to \infty} ( \frac{1}{m} [A^{n-k}] + \beta_m) \in \overline{P_k(X)}$. This proves (iii).

Finally assume that $X$ is smooth and let $\alpha \in \overline{P_k(X)}$. Let $\beta_m \in P_k(X)$ be such that 
\[ \alpha = \lim\limits_{m \to \infty} \beta_m. \]
Let $A_1,$ $\ldots,$ $A_{n-k}$ be any ample $\R$-Cartier $\R$-divisors on $X$. Then $[A_1 \cdots A_{n-k}] \in P_k(X)$, whence, by Theorem \ref{aperto}, for $m \gg 0$ we have that
\[ [A_1 \cdots A_{n-k}] + \alpha - \beta_m \in P_k(X). \]
But then 
\[ \alpha + [A_1 \cdots A_{n-k}] = ([A_1 \cdots A_{n-k}] + \alpha - \beta_m) + \beta_m \in P_k(X) \subseteq NSA_k(X)\]
by Lemma \ref{int} and therefore $\B_{\rm num} (\alpha + [A_1 \cdots A_{n-k}]) = \emptyset$. Hence $\B_-(\alpha) = \emptyset$ and (iv) is proved.
\end{proof}

\section{More properties of base loci of cycles} 
\label{cic}

When $X$ is smooth we can give more results.

First, Proposition \ref{tutte}\eqref{riduzgen} can be improved as follows.

\begin{prop}
\label{riduz}
Assume that $X$ is smooth and let $\alpha \in N_k(X)$. Then there is an $\varepsilon_{\alpha} > 0$ such that
\[ \B_{\rm num}(\alpha - \beta) \subseteq \B_+(\alpha) \]
for every $\beta \in N_k(X)$ such that $||\beta|| < \varepsilon_{\alpha}$ and
\[ \B_+(\alpha) = \B_{\rm num}(\alpha - [A_1 \cdots A_{n-k}]) \]
for every $A_1, \ldots, A_{n-k}$ ample $\R$-Cartier $\R$-divisors on $X$ such that $||[A_1 \cdots A_{n-k}]|| < \varepsilon_{\alpha}$.
\end{prop}
\begin{proof}
By Definition \ref{luoghibase} there are ample $\R$-Cartier $\R$-divisors $A_{i1}, \ldots, A_{i, n-k}$, $1 \leq i \leq s$ such that 
 \[ \B_+(\alpha)= \bigcap\limits_{i=1}^s \B_{\rm num}(\alpha - [A_{i1} \cdots A_{i, n-k}]). \]
For all $1 \leq i \leq s$, $[A_{i1} \cdots A_{i, n-k}] \in P_k(X)$, whence, by Lemma \ref{int}, there is an $\varepsilon_{\alpha} > 0$ (independent of $i$) such that $[A_{i1} \cdots A_{i, n-k}] - \beta \in NSA_k(X)$ for every $\beta \in N_k(X)$ such that $||\beta|| < \varepsilon_{\alpha}$. Then, using Proposition \ref{tutte}\eqref{unione-ii}, we have
\[ \B_{\rm num}(\alpha - \beta) = \B_{\rm num}(\alpha - [A_{i1} \cdots A_{i, n-k}] + [A_{i1} \cdots A_{i, n-k}] - \beta) \subseteq  \]
\[ \subseteq \B_{\rm num}(\alpha - [A_{i1} \cdots A_{i, n-k}]) \cup \B_{\rm num}([A_{i1} \cdots A_{i, n-k}] - \beta) =  \B_{\rm num}(\alpha - [A_{i1} \cdots A_{i, n-k}])\]
Therefore $\B_{\rm num}(\alpha - \beta) \subseteq \bigcap\limits_{i=1}^s \B_{\rm num}(\alpha - [A_{i1} \cdots A_{i, n-k}]) = \B_+(\alpha)$. 

Now if $A_1, \ldots, A_{n-k}$ are ample $\R$-Cartier $\R$-divisors on $X$ such that $||[A_1 \cdots A_{n-k}]|| < \varepsilon_{\alpha}$, then, by definition, $\B_+(\alpha) \subseteq \B_{\rm num}(\alpha - [A_1 \cdots A_{n-k}])$, giving the other inclusion. 
\end{proof}

\begin{cor}
\label{riduz+}
Assume that $X$ is smooth. Let $\alpha \in N_k(X)$ and let $\varepsilon_{\alpha}$ be as in Proposition \ref{riduz}. Then 
\[ \B_+(\alpha - \beta) \subseteq \B_+(\alpha) \]
for every $\beta \in N_k(X)$ such that $||\beta|| < \varepsilon_{\alpha}$ and equality holds if $\beta = [A_1 \cdots A_{n-k}]$ 
where $A_1, \ldots, A_{n-k}$ are ample $\R$-Cartier $\R$-divisors on $X$.
\end{cor}
\begin{proof}
Let $\beta \in N_k(X)$ be such that $||\beta|| < \varepsilon_{\alpha}$. Pick $A'_1, \ldots, A'_{n-k}$ ample $\R$-Cartier $\R$-divisors on $X$ such that 
\[ ||[A'_1 \cdots A'_{n-k}]|| < \min\{\varepsilon_{\alpha}- ||\beta||, \varepsilon_{\alpha - \beta}\}. \]
Then $\B_+(\alpha - \beta) = \B_{\rm num}(\alpha - \beta - [A'_1 \cdots A'_{n-k}])$ by Proposition \ref{riduz}. On the other hand,
$|| \beta + [A'_1 \cdots A'_{n-k}]|| < \varepsilon_{\alpha}$, whence $\B_{\rm num}(\alpha - \beta - [A'_1 \cdots A'_{n-k}]) \subseteq \B_+(\alpha)$ again by Proposition \ref{riduz} and therefore $\B_+(\alpha - \beta) \subseteq \B_+(\alpha)$. 

Now if $\beta = [A_1 \cdots A_{n-k}]$ where $A_1, \ldots, A_{n-k}$ are ample $\R$-Cartier $\R$-divisors on $X$ such that $||\beta|| < \varepsilon_{\alpha}$, then, by Proposition \ref{riduz} and Proposition \ref{tutte}\eqref{nest}

\vskip .4cm

\hskip .8cm $\B_+(\alpha - \beta) \subseteq \B_+(\alpha) \subseteq \B_{\rm num}(\alpha - [A_1 \cdots A_{n-k}]) = \B_{\rm num}(\alpha - \beta) \subseteq \B_+(\alpha - \beta).$
\end{proof}

The following is the analogue of \cite[Prop.~1.21]{elmnp1}.

\begin{prop}
\label{121}
Assume that $X$ is smooth. Let $\alpha \in N_k(X)$ and let $\varepsilon_{\alpha}$ be as in Proposition \ref{riduz}. Then 
\[ \B_-(\alpha - [A_1 \cdots A_{n-k}]) = \B_+(\alpha - [A_1 \cdots A_{n-k}]) = \B_+(\alpha) \]
for every $A_1, \ldots, A_{n-k}$ ample $\R$-Cartier $\R$-divisors on $X$ such that $||[A_1 \cdots A_{n-k}]|| < \varepsilon_{\alpha}$.
\end{prop}
\begin{proof}
For $i = 1, \ldots, n-k$, let $A'_i = (\frac{1}{2})^{\frac{1}{n-k}}A_i$ and let $\beta = \frac{1}{2} [A_1 \cdots A_{n-k}] = [A'_1 \cdots A'_{n-k}]$. Then Corollary \ref{riduz+} and Proposition \ref{tutte}\eqref{unb+}-\eqref{nest} give
\[ \B_+(\alpha) = \B_+(\alpha - \beta) =  \B_+(\alpha - 2 \beta + \beta) \subseteq \B_-(\alpha - 2 \beta) \subseteq \B_+(\alpha - 2 \beta) =  \B_+(\alpha) \]
whence $\B_-(\alpha - 2 \beta) = \B_+(\alpha - 2 \beta) =  \B_+(\alpha)$.
\end{proof}

\section{Stable cycles} 
\label{st}

In \cite[\S 1]{elmnp1} stable divisors were defined and studied. We prove some analogues for cycles.

\begin{defi}
Let $\alpha \in N_k(X)$. We say that $\alpha$ is {\it stable} if $\B_-(\alpha) = \B_+(\alpha)$. 
\end{defi}

As in \cite[Prop.~1.24]{elmnp1}, we can give several properties equivalent to stability.

\begin{prop}
\label{stab}
Assume that $X$ is smooth, let $\alpha \in N_k(X)$ and assume that the base field is uncountable. The following are equivalent:
\begin{itemize}
\item[(i)] $\alpha$ is stable; 
\item[(ii)] there are ample $\R$-Cartier $\R$-divisors $A_1, \ldots, A_{n-k}$ such that 

$\B_+(\alpha) = \B_+(\alpha + [A_1 \cdots A_{n-k}])$;
\item[(iii)] there is an $\varepsilon > 0$ such that $\B_+(\alpha) = \B_+(\alpha + \beta)$
for every $\beta \in N_k(X)$ such that $||\beta|| < \varepsilon$;
\item[(iv)] there is an $\varepsilon > 0$ such that $\B_-(\alpha) = \B_-(\alpha + \beta)$
for every $\beta \in N_k(X)$ such that $||\beta|| < \varepsilon$;
\item[(v)] there is an $\varepsilon > 0$ such that $\B_{\rm num}(\alpha + \beta) = \B_{\rm num}(\alpha + \beta')$
for every $\beta, \beta' \in N_k(X)$ such that $||\beta|| < \varepsilon$, $||\beta'|| < \varepsilon$.
\end{itemize}
\end{prop}
\begin{proof}
Assume that (i) holds and pick $A'_1, \ldots, A'_{n-k}$ ample $\R$-Cartier $\R$-divisors on $X$. Since $\B_-(\alpha) = \B_+(\alpha)$ is closed, by Proposition \ref{unb+num} we can find an $m \in \NN^+$ such that $\B_-(\alpha) = \B_+(\alpha + \frac{1}{m} [A'_1 \cdots A'_{n-k}])$. Setting $A_i = (\frac{1}{m})^{\frac{1}{n-k}} A'_i$ we get (ii). 

Now assume (ii). By Corollary \ref{riduz+} and Theorem \ref{aperto}, there is an $\varepsilon > 0$ such that
\[ \B_+(\alpha + \beta) \subseteq \B_+(\alpha) \ \mbox{and} \  [A_1 \cdots A_{n-k}] - \beta \in P_k(X) \] 
for every $\beta \in N_k(X)$ such that $||\beta|| < \varepsilon$. By Proposition \ref{tutte}\eqref{unione2} and Theorem \ref{b+}(ii) we get
\[ \B_+(\alpha) = \B_+(\alpha + [A_1 \cdots A_{n-k}]) \subseteq \B_+(\alpha + \beta) \cup \B_+([A_1 \cdots A_{n-k}] - \beta) = \B_+(\alpha + \beta) \subseteq \B_+(\alpha) \]
whence (iii). 

Suppose now (iii) holds and let $\beta \in N_k(X)$ such that $||\beta|| < \varepsilon$. We first prove that
\begin{equation}
\label{a}
\B_-(\alpha + \beta) =  \bigcup\limits_{A_1, \ldots, A_{n-k}} \B_+(\alpha + \beta + [A_1 \cdots A_{n-k}])
\end{equation}
where $A_1, \ldots, A_{n-k}$ run among all ample $\R$-Cartier $\R$-divisors on $X$ such that
\[ ||[A_1 \cdots A_{n-k}]|| < \varepsilon - ||\beta||. \] 
To see \eqref{a} let $A'_1, \ldots, A'_{n-k}$ be any ample $\R$-Cartier $\R$-divisors on $X$. By Theorem \ref{aperto} we can choose small ample $\R$-Cartier $\R$-divisors $A''_1, \ldots, A''_{n-k}$ such that
\[ ||[A''_1 \cdots A''_{n-k}]|| < \varepsilon - ||\beta|| \ \mbox{and} \ [A'_1 \cdots A'_{n-k}] - [A''_1 \cdots A''_{n-k}] \in P_k(X). \] 
By Proposition \ref{tutte}\eqref{unione-i}-\eqref{unione-ii}-\eqref{nest} we get
\[ \B_{\rm num}(\alpha + \beta + [A'_1 \cdots A'_{n-k}]) \subseteq \B_{\rm num}(\alpha + \beta + [A''_1 \cdots A''_{n-k}]) \cup \B_{\rm num}([A'_1 \cdots A'_{n-k}] - [A''_1 \cdots A''_{n-k}]) = \]
\[ = \B_{\rm num}(\alpha + \beta + [A''_1 \cdots A''_{n-k}]) \subseteq \B_+(\alpha + \beta + [A''_1 \cdots A''_{n-k}]) \]
whence we get the inclusion "$\subseteq$" in \eqref{a}. Now let $A_1, \ldots, A_{n-k}$ be any ample $\R$-Cartier $\R$-divisors on $X$ such that $||[A_1 \cdots A_{n-k}]|| < \varepsilon - ||\beta||$. Let $m \gg 0$ be such that 
\[ ||\frac{1}{m}[A_1 \cdots A_{n-k}]|| < \varepsilon_{\alpha + \beta + [A_1 \cdots A_{n-k}]}. \]
By Proposition \ref{riduz} we have
\[ \B_+(\alpha + \beta + [A_1 \cdots A_{n-k}]) = \B_{\rm num}(\alpha + \beta +[A_1 \cdots A_{n-k}] - \frac{1}{m}[A_1 \cdots A_{n-k}]) = \]
\[ = \B_{\rm num}(\alpha + \beta + \frac{m-1}{m}[A_1 \cdots A_{n-k}]) \subseteq  \B_-(\alpha + \beta) \]
and this proves the inclusion "$\supseteq$" in \eqref{a}, thus giving \eqref{a}. On the other hand, for every ample $\R$-Cartier $\R$-divisors $A_1, \ldots, A_{n-k}$ such that $||[A_1 \cdots A_{n-k}]|| < \varepsilon - ||\beta||$ we have that $||\beta + [A_1 \cdots A_{n-k}]|| < \varepsilon$, whence, by (iii), $\B_+(\alpha + \beta + [A_1 \cdots A_{n-k}]) = \B_+(\alpha)$ and therefore \eqref{a} gives that
\begin{equation}
\label{b}
\B_-(\alpha + \beta) = \B_+(\alpha) \ \mbox{for every} \ \beta \in N_k(X) \ \mbox{such that} \ ||\beta|| < \varepsilon.
\end{equation}
In particular this holds for $\beta = 0$, so that $\B_-(\alpha) = \B_+(\alpha)$ and \eqref{b} gives (iv). 

Assume (iv) and let $\varepsilon' > 0$ be such that $\varepsilon' \leq \min\{\varepsilon, \varepsilon_{\alpha}\}$. Let $\beta \in N_k(X)$ such that $||\beta|| < \varepsilon'$. Let $A_1, \ldots, A_{n-k}$ be sufficiently small ample $\R$-Cartier $\R$-divisors so that $||[A_1 \cdots A_{n-k}]|| < \varepsilon'$. By (iv) and Proposition \ref{121} we get
\[ \B_-(\alpha) = \B_-(\alpha - [A_1 \cdots A_{n-k}]) = \B_+(\alpha). \]
Then, by (iv), Propositon \ref{tutte}\eqref{nest} and Corollary \ref{riduz+} we find
\[ \B_+(\alpha) = \B_-(\alpha) =  \B_-(\alpha + \beta) \subseteq \B_{\rm num}(\alpha + \beta) \subseteq \B_+(\alpha + \beta) \subseteq \B_+(\alpha) \]
whence $\B_{\rm num}(\alpha + \beta) = \B_+(\alpha)$ for every $\beta \in N_k(X)$ such that $||\beta|| < \varepsilon'$ and (v) holds.

Finally assume (v) and let $A_1, \ldots, A_{n-k}$ be sufficiently small ample $\R$-Cartier $\R$-divisors so that $||[A_1 \cdots A_{n-k}]|| < \min\{\varepsilon, \varepsilon_{\alpha}\}$. Now Proposition \ref{riduz}, (v) and Propositon \ref{tutte}\eqref{nest} give
\[ \B_+(\alpha) = \B_{\rm num}(\alpha - [A_1 \cdots A_{n-k}]) = \B_{\rm num}(\alpha + [A_1 \cdots A_{n-k}]) \subseteq \B_-(\alpha) \subseteq \B_+(\alpha) \]
and this gives (i).
\end{proof} 

\begin{remark} 
\label{altro}
Let $\alpha \in N_k(X)$. If $\alpha$ is not pseudoeffective, then it is stable by Theorem \ref{b-}(i) and Propositon \ref{tutte}\eqref{nest}. If $\alpha$ is pseudoeffective but not big, and the base field is uncountable, then it is not stable by Theorems \ref{b+}(i) and \ref{b-}(ii).
\end{remark}

\begin{cor}
\label{dense}
Assume that $X$ is smooth and that the base field is uncountable. Then the cone of stable classes is open and dense in $N_k(X)$.
\end{cor}
\begin{proof}
It is open by Proposition \ref{stab} and dense by Proposition \ref{121}.
\end{proof}

\end{document}